\documentclass[letterpaper, 11 pt]{article}
\usepackage{fullpage,amsthm, amsmath, amssymb, amsfonts}
\usepackage{graphicx}

\newcommand{\ZZ}{\mathbb{Z}}

\newcommand{\RR}{\mathbb{R}}

\newcommand{\PP}{\mathcal{P}}

\newcommand{\newword}[1]{\textbf{\emph{#1}}}

\newcommand{\OO}{\mathcal{O}}

\newcommand{\M}{\mathcal{M}}

\newcommand{\T}{\mathcal{T}}

\newcommand{\SSS}{\mathcal{S}}
\newcommand{\SSSS}{\mathcal{A}}

\def\M{\mathcal{M}}

\newtheorem{conj}{Conjecture}[section]

\newtheorem{theorem}[conj]{Theorem}

\newtheorem{proposition}[conj]{Proposition}

\newtheorem{lemma}[conj]{Lemma}

\newtheorem{corollary}[conj]{Corollary}
\newtheorem{question}[conj]{Question}

\theoremstyle{definition}
\newtheorem{definition}[conj]{Definition}

\begin{document}
 
\title{Triangulations of $\Delta_{n-1} \times \Delta_{d-1}$ and Matching Ensembles}
\author{Suho Oh, Hwanchul Yoo }



\date{}
\maketitle

\abstract{We suggest an axiom system for a collection of matchings that describes the triangulation of product of simplices.}

\section{Introduction}

Triangulations of a product of two simplices are beautiful and important objects. In this paper, we describe a new combinatorial model for describing the triangulations of product of two simplices, $\Delta_{n-1} \times \Delta_{d-1}$. 

 The model we use is called \newword{matching ensemble}; it was motivated by \newword{matching fields} introduced and studied by Bernstein and Zelevinsky \cite{MF}. The $(n,d)$-matching field, for $n \geq d$, is a collection of bijections (``matchings'') between $d$-element subsets of $[n]=\{1,\ldots,n\}$ and the set $[d]=\{1,\ldots,d\}$. A matching field is \newword{coherent} if it satisfies  \newword{linkage property}, which is similar to the basis exchange axiom for matroids. These objects were used to study the \newword{Newton polytope} of the product of all maximal minors of an $n$-by-$d$ matrix of indeterminates.  
 
 We define an $(n,d)$-matching ensemble as a collection of matchings between subsets of $[n]$ and subsets of $[d]$ such that:
\begin{itemize}
 \item there is one matching between every pair with same cardinality, 
 \item a submatching of any matching is in the collection, and 
 \item the matchings contained in the collection satisfy the linkage property. 
\end{itemize}

We show that there is a bijection between $(n,d)$-matching ensembles and triangulations of $\Delta_{n-1} \times \Delta_{d-1}$.

\section{Triangulation of $\Delta_{n-1} \times \Delta_{d-1}$}

Following \cite{Postnikov01012009}, we will study $\Delta_{n-1} \times \Delta_{d-1}$ using a class of polytopes associated to bipartite graphs $G \subseteq K_{n,d}$, called \newword{root polytopes}. We think of the complete bipartite graph $K_{n,d}$ as having a set of left vertices $\{1,\ldots,n\}$ and a set of right vertices $\{\bar{1},\ldots,\bar{d}\}$. We define $Q_G$ to denote the convex hull of points $e_i - e_{\bar{j}}$ for edges $(i,\bar{j})$ of $G$ where $e_1,\ldots,e_n,e_{\bar{1}},\ldots,e_{\bar{d}}$ are the coordinate vectors in $\RR^{n+d}$. When $G$ is the complete bipartite graph $K_{n,d}$, the polytope $Q_G$ is exactly $\Delta_{n-1} \times \Delta_{d-1}$.

\begin{definition}
A \newword{triangulation} of a polytope $P$ is a subdivision of $P$ into a union of simplices of the same dimension as $P$ such that each simplex is the convex hull of some subset of vertices of $P$ and any two simplices intersect properly, i.e., the intersection of any two simplices is their common face.
\end{definition}

\begin{figure}[htbp]
	\centering
		\includegraphics[width=0.5\textwidth]{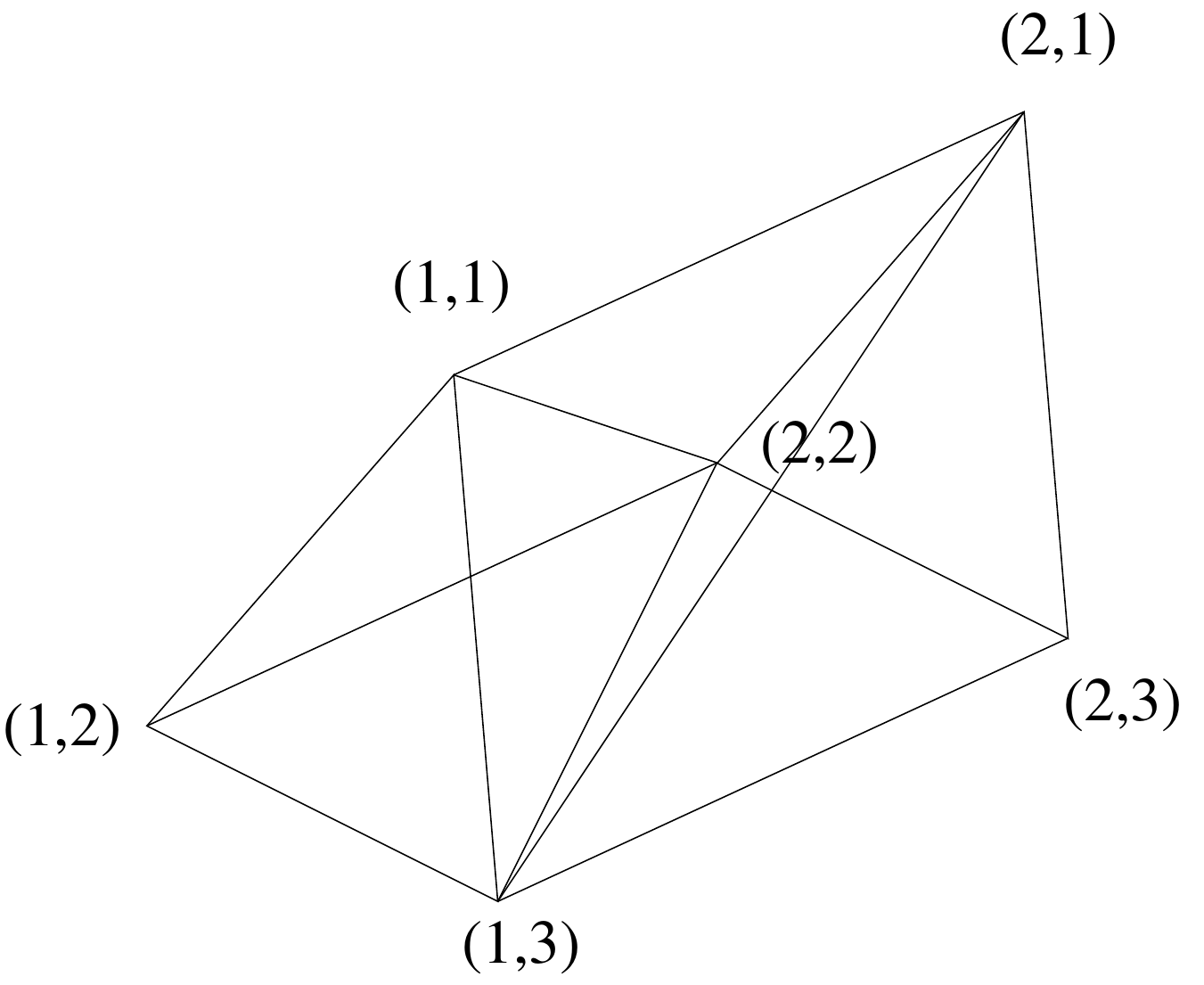}
	\caption{An example of a triangulation of $\Delta_1 \times \Delta_2$.}
	\label{fig:d3d3-v01}
\end{figure}

Figure~\ref{fig:d3d3-v01} is an example of a triangulation of $\Delta_1 \times \Delta_2$. We will be studying the triangulation of $\Delta_{n-1} \times \Delta_{d-1}$. The simplices of the triangulations can be described via spanning trees due to the following lemma:

\begin{lemma}[Lemma 12.5 of \cite{Postnikov01012009}]
\label{lem:simplextree}
For a subgraph $H \subseteq K_{n,d}$, the polytope $Q_H$ is a $n+d-2$ dimensional simplex if and only if $H$ is a spanning tree of $K_{n,d}$. All $n+d-2$ dimensional simplices of this form have the same volume $\frac{1}{(n+d-2)!}$.
\end{lemma}

\begin{figure}[htbp]
	\centering
		\includegraphics[width=0.3\textwidth]{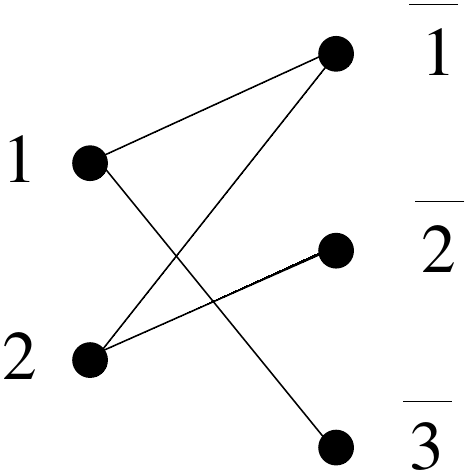}
	\caption{A spanning tree of $K_{2,3}$ which encodes one of the simplices in Figure~\ref{fig:d3d3-v01}.}
	\label{fig:d3d3-tree}
\end{figure}

If we look at the triangulation in Figure~\ref{fig:d3d3-v01}, one of the simplices is a convex hull of $e_1 - e_{\bar{1}},e_1 - e_{\bar{3}},e_2 - e_{\bar{1}}, e_2 - e_{\bar{2}}$. This is encoded as a spanning tree in Figure~\ref{fig:d3d3-tree}.

Lemma~\ref{lem:simplextree} tells us that a triangulation of $\Delta_{n-1} \times \Delta_{d-1}$ is a collection of simplices $\{Q_{T_1},\ldots,Q_{T_s}\}$, for some spanning trees $T_1,\ldots,T_s$ of $K_{n,d}$ such that $Q_{K_{n,d}} = \cup Q_{T_i}$ and each intersection $Q_{T_i} \cap Q_{T_j}$ is the common face of the two simplices. Lemma~\ref{lem:simplextree} combined with Theorem 12.2 of \cite{Postnikov01012009}, implies that:

\begin{lemma}
\label{lem:triangcardinality}
A triangulation of $\Delta_{n-1} \times \Delta_{d-1}$ contains exactly $|H_{n-1,d}|$ number of $n+d-2$ dimensional simplices, where $H_{n,d}$ is defined as the set $\{(a_1,\ldots,a_d)| \sum_{i=1}^d a_i = n, a_i \in \ZZ^{\geq 0} \}$.
\end{lemma}

For two spanning trees $T$ and $T'$ of $K_{n,d}$, let $U(T,T')$ be union of edges of $T$ and $T'$ with edges of $T$ oriented from left to right and edges of $T'$ oriented from right to left. A directed cycle is a sequence of directed edges $(i_1,i_2),(i_2,i_3),\ldots,(i_{k-1},i_k),(i_k,i_1)$ such that all $i_1,\ldots,i_k$ are distinct. We say that $T$ and $T'$ are \newword{compatible} if the directed graph $U(T,T')$ has no directed cycles of length $\geq 4$, and \newword{incompatible} if not.

\begin{lemma}[Lemma 12.6 of \cite{Postnikov01012009}]
\label{lem:triangcompatible}
For two trees $T$ and $T'$, the intersection $Q_T \cap Q_{T'}$ is a common face of the simplices $Q_T$ and $Q_{T'}$ if and only if $T$ and $T'$ are compatible.
\end{lemma}


We will be using left-degree vectors and right-degree vectors which are useful for analyzing the simplices.

\begin{definition}[Definition 12.7 of \cite{Postnikov01012009}]
For a spanning tree $T \subseteq K_{n,d}$, we define the \newword{left degree vector} $LD(T) = (a_1,\ldots,a_n)$ and \newword{right degree vector} $RD(T) = (a_{\bar{1}},\ldots,a_{\bar{d}})$, where $a_i$ and $a_{\bar{j}}$ are the degrees of the vertices $i$ and $\bar{j}$ in $T$ minus $1$.
\end{definition}

\begin{lemma}[Lemma 12.7 of \cite{Postnikov01012009}]
\label{lem:triangDV}
Let $\{Q_{T_1},\ldots,Q_{T_s}\}$ be a triangulation of $K_{n,d}$. Then for $i \not = j$, $T_i$ and $T_j$ have different left degree vectors and different right degree vectors.
\end{lemma}

Combining Lemma~\ref{lem:triangcardinality}, Lemma~\ref{lem:triangcompatible} and Lemma~\ref{lem:triangDV}, we can conclude the following:

\begin{proposition}
\label{prop:triangfill}
Let $\T = \{Q_{T_1},\ldots,Q_{T_s}\}$ be a set of polytopes where $T_1,\ldots,T_s$ are spanning trees of $K_{n,d}$. This set defines a triangulation of $\Delta_{n-1} \times \Delta_{d-1}$ if and only if the following conditions are satisfied:
\begin{itemize}
\item for any $i \not =  j$, $T_i$ and $T_j$ are compatible and,
\item the map $T \rightarrow RD(T)$ is a bijection between $\T$ and $H_{n-1,d}$.
\end{itemize}
\end{proposition}

\section{Tropical oriented matroids}

In this section, we will go over tropical oriented matroids, which is useful for studying triangulations of $\Delta_{n-1} \times \Delta_{d-1}$.

\begin{definition}
An $(n,d)$-type is an $n$-tuple $A=(A_1,\ldots,A_n)$ of nonempty subsets of $[d]:=\{1,\cdots,d\}$. The sets $A_1,\cdots,A_n$ are called the \newword{coordinates} of $A$. 
\end{definition}

We can think of $(n,d)$-types as subgraphs of $K_{n,d}$. Given an $(n,d)$-type $A=(A_1,\ldots,A_n)$, let $G_A$ be a subgraph of $K_{n,d}$ consisting of edges $(i,\bar{j})$ for each $j \in A_i$. We will say that an $(n,d)$-type $A$ is \newword{generic} if $G_A$ does not contain a cycle. Given an $(n,d)$-type $A$, we say that an $(n,d)$-type $B$ is a \newword{refinement} of $A$ if $G_B$ is a subgraph of $G_A$. We say that two $(n,d)$-types $A$ and $B$ are \newword{compatible} if their corresponding graphs $G_A$ and $G_B$ are compatible. There is an easy way to check if two types $A$ and $B$ are compatible or not:

\begin{lemma}
\label{lem:typecomp}
Let $A$ and $B$ be $(n,d)$-types. We say that there is a \newword{cycle} of length $k$ between $A$ and $B$ if after some relabeling of the set $[n]$, we have $i_1 \in A_1,B_k,i_2 \in A_2,B_1,\ldots,i_k \in A_k,B_{k-1}$, where $k \geq 2$ and $i_1,\ldots,i_k$ are all distinct. Then $A$ and $B$ are compatible if and only if there is no cycle between $A$ and $B$.
\end{lemma}
\begin{proof}
A cycle of length $k$ between $A$ and $B$, given by $i_1 \in A_1,B_k,i_2 \in A_2,B_1,\ldots,i_k \in A_k,B_{k-1}$, corresponds to an alternating cycle of length $2k$ in $U(G_A,G_B)$, given by $(\bar{i_1},k),(k,\bar{i_k}),(\bar{i_k},k-1),\ldots,(\bar{i_2},1),(1,\bar{i_1})$. Hence there is a cycle of length $k$ between $A$ and $B$ if and only if there is a cycle of length $2k$ in $U(G_A,G_B)$, which implies that there is a cycle between $A$ and $B$ if and only if $A$ and $B$ are incompatible.
\end{proof}

Now we may define a generic tropical oriented matroid, which is a collection of generic types, satisfying four axioms.

\begin{definition}[\cite{2007arXiv0706.2920A}]
A \newword{generic tropical oriented matroid} $\OO$ (with parameters $(n,d)$) is a collection of generic $(n,d)$-types which satisfy the following four axioms:
\begin{itemize}
\item Boundary : For each $j \in [d]$, the type $\newword{j} :=(j,\cdots,j)$ is in $\OO$.
\item Elimination : If we have two types $A$ and $B$ in $\OO$ and a position $j \in [n]$, then there exists a type $C$ in $\OO$ with $C_j = A_j \cup B_j$, and $C_k \in \{A_k,B_k,A_k \cup B_k\}$ for all $k \in [n]$.
\item Comparability : For any two types $A$ and $B$, they are compatible.
\item Surrounding : If $A$ is a type in $\OO$, then any refinement of $A$ is also in $\OO$.
\end{itemize}
\end{definition}

It was conjectured in \cite{2007arXiv0706.2920A} and proved in \cite{MR2820754} that generic tropical oriented matroids (with parameters $(n,d))$ are in bijection with triangulations of $\Delta_{n-1} \times \Delta_{d-1}$. Later, Horn \cite{MR2957992} showed that the same relationship holds for tropical oriented matroids and subdivisions. Unless otherwise stated, all the tropical oriented matroids we use in this paper will be generic tropical oriented matroids.

We call a type $A$, where $G_A$ is a spanning tree of $K_{n,d}$, as a \newword{tree-type} (In \cite{2007arXiv0706.2920A}, the word \newword{vertex} is used for such types).

\begin{theorem}[\cite{2007arXiv0706.2920A},\cite{MR2820754}]
The tree-types of a tropical oriented matroid $\OO$ completely determine it. The set $\{Q_{G_{A_1}},\ldots,Q_{G_{A_s}}\}$ describes a triangulation of $\Delta_{n-1} \times \Delta_{d-1}$ if and only if $A_1,\ldots,A_s$ are tree-types of a generic tropical oriented matroid (with parameters $(n,d)$).
\end{theorem}

For example, in the triangulation of Figure~\ref{fig:d3d3-v01}, the $3$ simplices are encoded as the tree-types $(13,12),(123,2),(3,123)$. The tropical oriented matroid corresponding to the collection obtained by refining these types is $\{(13,12),(13,2),(13,1),(1,12),(3,12),(1,1),\ldots\}$.

A \newword{tope} is a type $A = (A_1,\ldots,A_n)$ such that all $A_i$ are singletons. A refinement $B$ of $A$ is called a \newword{total refinement} of $A$ if $B$ is a tope.

\begin{theorem}[\cite{2007arXiv0706.2920A}]
\label{thm:topestom}
The topes of a tropical oriented matroid $\OO$ completely determine it. To be precise,
$A = (A_1,\ldots,A_n)$ is in $\OO$ if and only if the following two conditions hold:
\begin{itemize}
\item $A$ satisfies the compatibility axiom with every tope of $\OO$. 
\item All of $A$'s total refinements are topes of M.
\end{itemize}
\end{theorem}

The tropical oriented matroid corresponding to the triangulation of Figure~\ref{fig:d3d3-v01} has the following topes : $\{(1,1),(1,2),(2,2),(3,1),(3,2),(3,3)\}$.

A more natural way to think about $(n,d)$-types is to think in terms of mixed subdivisions of $n\Delta_{d-1}$. Via the Cayley trick, one can think of a triangulation of $\Delta_{n-1} \times \Delta_{d-1}$ as a fine mixed subdivision of $n \Delta_{d-1}$ \cite{SantosCayley}. An example of this is shown in Figure~\ref{fig:d3d3-v02}.

\begin{figure}[htbp]
	\centering
		\includegraphics[width=0.6\textwidth]{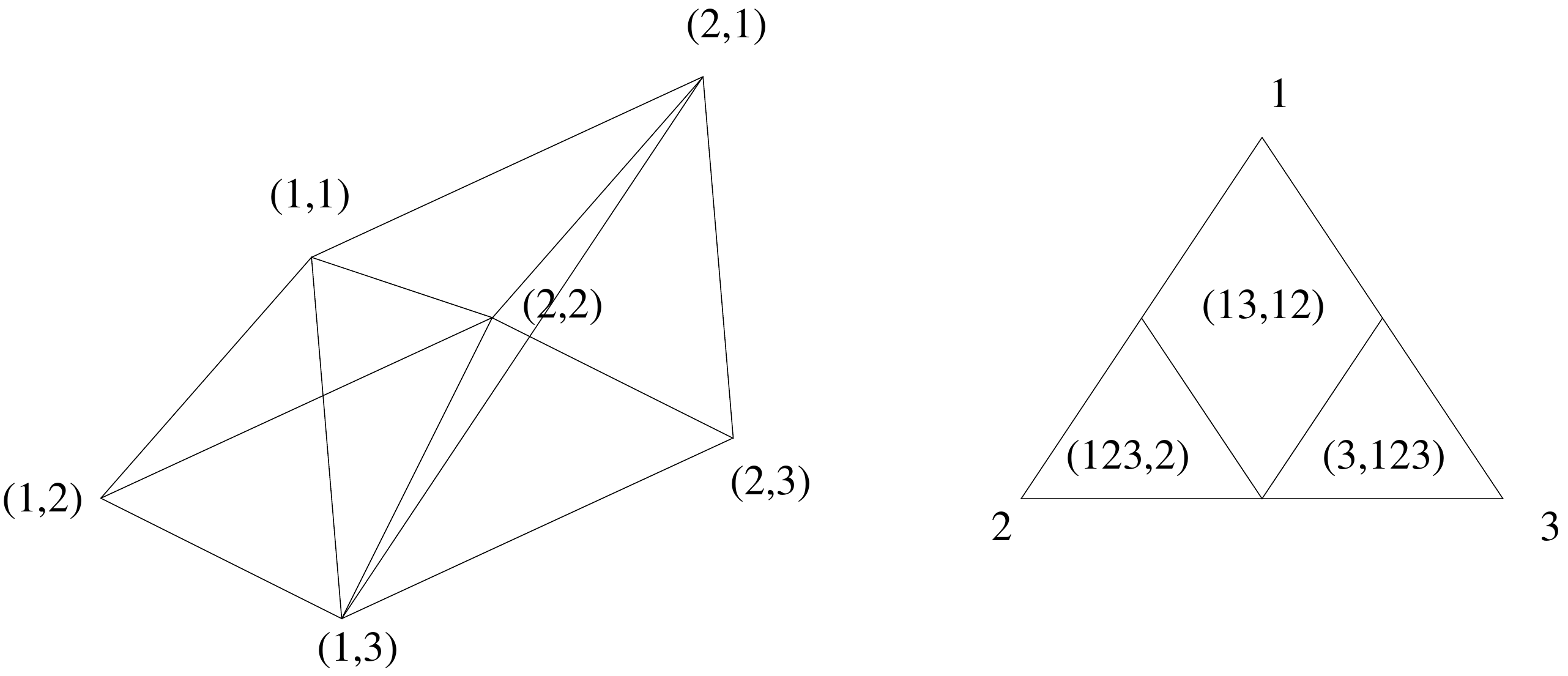}
	\caption{A triangulation of $\Delta_1 \times \Delta_2$, and the corresponding mixed subdivision of $2\Delta_2$.}
	\label{fig:d3d3-v02}
\end{figure}

\begin{definition}[\cite{Postnikov01012009}]
Let $r$ be the dimension of the Minkowski sum $P_1 + \cdots + P_n$. A \newword{Minkowski cell} in this sum is a polytope $B_1 + \cdots + B_n$ of dimension $r$ where $B_i$ is the convex hull of some subset of vertices of $P_i$. A \newword{mixed subdivision} of the sum is the decomposition into union of Minkowski cells such that intersection of any two cells is their common face. A mixed subdivision is \newword{fine} if there is no refinement possible.
\end{definition}

Let $e_1,\ldots,e_d$ be the coordinate vectors of $\RR^d$. We use $\Delta_I$ to denote the convex hull of $e_i$'s for $i \in I$. We are studying the fine mixed subdivision of $n\Delta_{[d]}$. Then for an $(n,d)$-type $A$, the polytope $\Delta_{A_1} + \cdots + \Delta_{A_n}$ is a fine mixed cell if and only if $G_A$ is a spanning tree of $K_{n,d}$ \cite{Postnikov01012009}.

Hence each tope $A = (\{a_1\},\ldots,\{a_n\})$ can be thought as a point $\Delta_{\{a_1\}} + \cdots + \Delta_{\{a_n\}}$ which is a integer lattice point of the dilated simplex $n\Delta_{[d]}$. The actual coordinate of this point $(b_1,\ldots,b_d)$ in $\RR^d$, where $b_i$ counts the number of times $i \in [d]$ occurs among $a_j$'s, will be denoted as $pos(A)$, the position of $A$.

\begin{lemma}
\label{lem:possame}
Let $A$ and $B$ be $(n,d)$-topes such that $pos(A) = pos(B)$. If they are compatible, then $A=B$.
\end{lemma}
\begin{proof}
Delete the coordinates such that $A_i = B_i$. Assume for sake of contradiction that $A \not = B$ and $A$ and $B$ are not compatible. Starting from an arbitrary coordinate, we can find a cycle between $A$ and $B$, just by following the elements of the corresponding coordinates.
\end{proof}

\begin{lemma}
\label{lem:topesoftom}
Let $\OO$ be a tropical oriented matroid (with parameters $(n,d)$) and let $\hat{\OO}$ denote the set of topes of $\OO$. The map $A \rightarrow pos(A)$ gives a bijection between $\hat{\OO}$ and $H_{n,d}$.
\end{lemma}
\begin{proof}
The map being one-to-one follows from comparability and Lemma~\ref{lem:possame}. To show that the map is onto, we will show the following claim: Given a tree-$(n,d)$-type $B$ with right degree vector $a=(a_1,\ldots,a_d)$, for each $i \in [d]$, we can find a tope $A$ with $pos(A) = a + e_i$ by refining $T$. Combined with Proposition~\ref{prop:triangfill}, this claim is enough to conclude that the map is onto.

Let us first fix $i \in [d]$. For each $j \in [n]$, there exists a unique element $x_j \in B_j$ such that to go from a left vertex $j$ to right vertex $\hat{i}$ in $G_B$, one has to pass through the vertex $\hat{x_j}$. Set $A_j = \{x_j\}$ for all $j \in [n]$ to get a tope $A$. Then all elements of $[d]$ except for $i$ occurs exactly once among $B_j \setminus A_j$'s. Hence we may conclude that $pos(A) = RD(B) + e_i$, which proves the claim.
\end{proof}

The left image in Figure~\ref{fig:msub-topes-v01} gives the set of topes in a tropical oriented matroid with parameters $(2,3)$. The right image emphasizes the fact that we get a bijection between those topes and points of $H_{2,3}$ by the map $A \rightarrow pos(A)$.

\begin{figure}[htbp]
	\centering
		\includegraphics[width=0.4\textwidth]{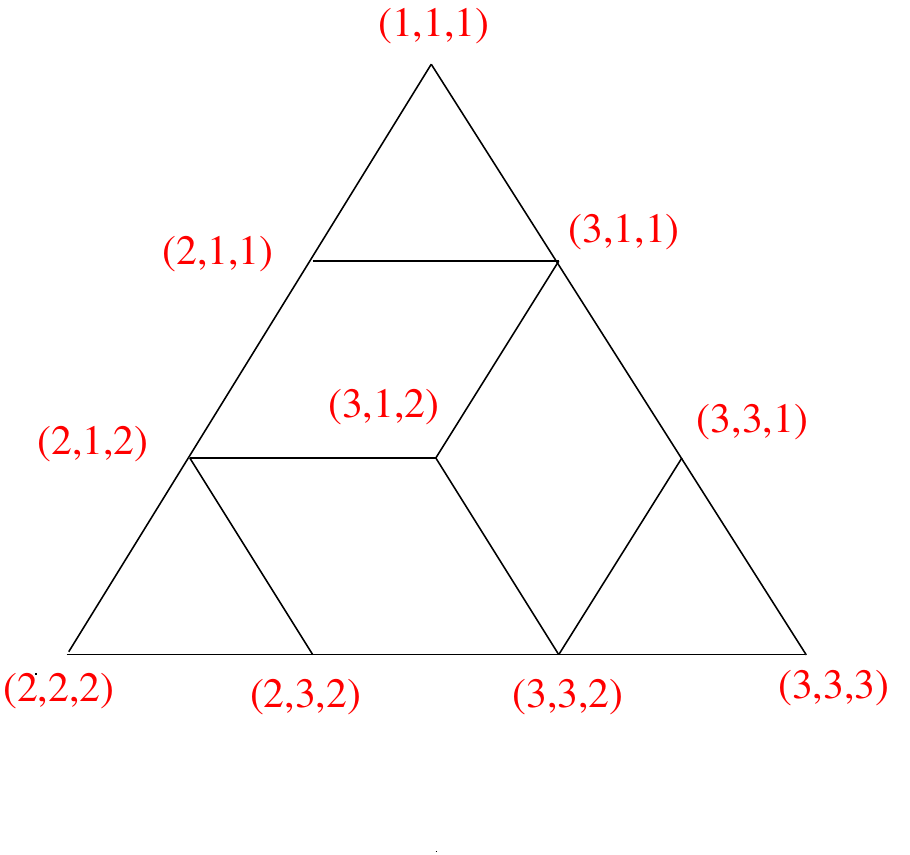}
		\includegraphics[width=0.4\textwidth]{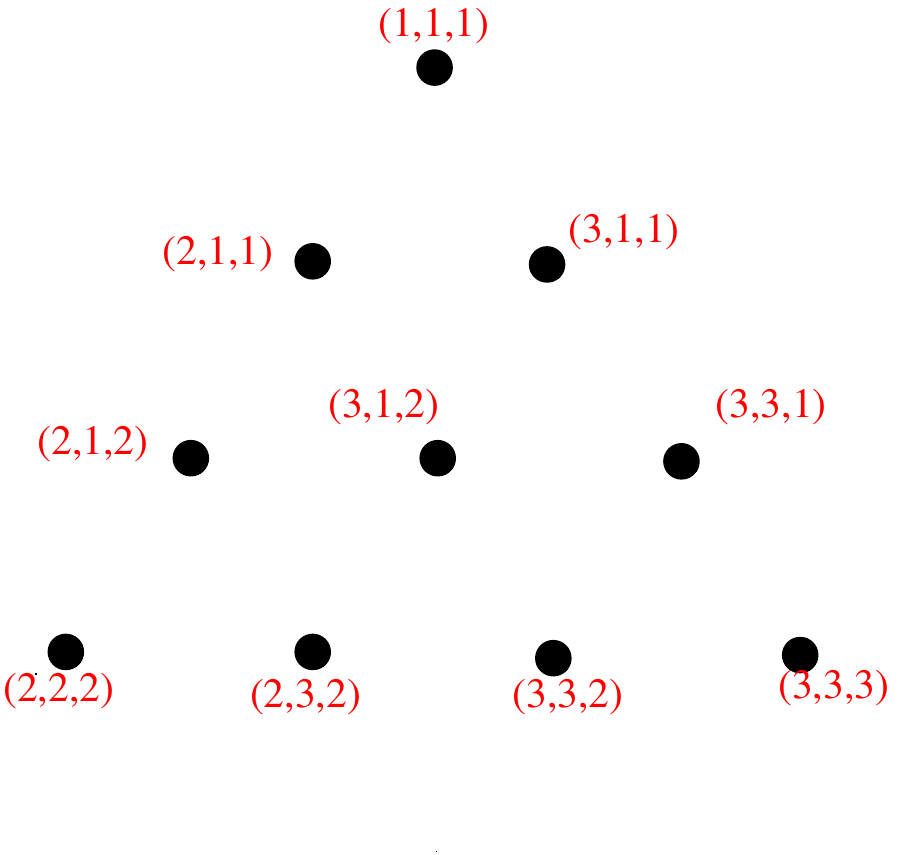}
	\caption{Set of topes of a tropical oriented matroid, and their bijection with $H_{4,3}$.}
	\label{fig:msub-topes-v01}
\end{figure}

For each element $a=(a_1,\ldots,a_d)$ of $H_{n-1,d}$, we can associate a simplex in $H_{n,d}$ consisting of vertices $a+e_i$ for each $i \in [d]$. We call such simplex a \newword{unit lattice simplex} at $a$, and denote the points $a+e_i$ as \newword{base points} of the simplex. Moreover, for a simplex $Q_T$ of a triangulation of $\Delta_{n-1} \times \Delta_{d-1}$, we denote the base points of the unit lattice simplex at $RD(T)$ as the base points of $Q_T$. We denote a tope that is positioned at a base point of $Q_T$ as a \newword{base tope} of $T$.


\begin{corollary}
\label{cor:tombaseunion}
Let $\OO$ be a tropical oriented matroid (with parameters $(n,d)$) and let $\hat{\OO}$ denote the set of topes of $\OO$. Fix $a \in H_{n-1,d}$ and consider the unit lattice simplex at $a$. Let $A^1,\ldots,A^d$ be the topes corresponding to the base points of the unit lattice simplex. Set $B$ to be the $(n,d)$-type obtained by taking the union of $A^1,\ldots,A^d$ (In other words, for each $j \in [n]$, $B_j = \cup_{k=1}^n A_j^k$). Then $B$ describes a spanning tree which has right-degree vector equal to $a$.
\end{corollary}
\begin{proof}
From the claim inside the proof of Lemma~\ref{lem:topesoftom}, it is enough to show that $B$ is a spanning tree. Since $\Delta_{B_1} + \cdots + \Delta_{B_n}$ has to contain the points corresponding to $A^j$'s, it has to be $d-1$-dimensional. For this to happen, $G_B$ has to be spanning $K_{n,d}$ and be connected. Moreover, $G_B$ has to be contained in the spanning tree $T$ which corresponds to the simplex having right degree vector given by $a$. Therefore, $B$ has to equal $T$, which is a spanning tree of $K_{n,d}$.

\end{proof}

Let us look at Figure~\ref{fig:msub-trees}. In the right image, the tree-types are placed according to their right-degree vector. If we look at the base topes $(3,1,1),(3,1,2),(3,3,2)$, which corresponds to the base points of a simplex at $(1,0,1)$, their union is the type $(3,13,12)$, which is exactly the tree-type of the tropical oriented matroid having right degree vector $(1,0,1)$.

\begin{figure}[htbp]
	\centering
		\includegraphics[width=0.4\textwidth]{msub-topes-v01}
		\includegraphics[width=0.4\textwidth]{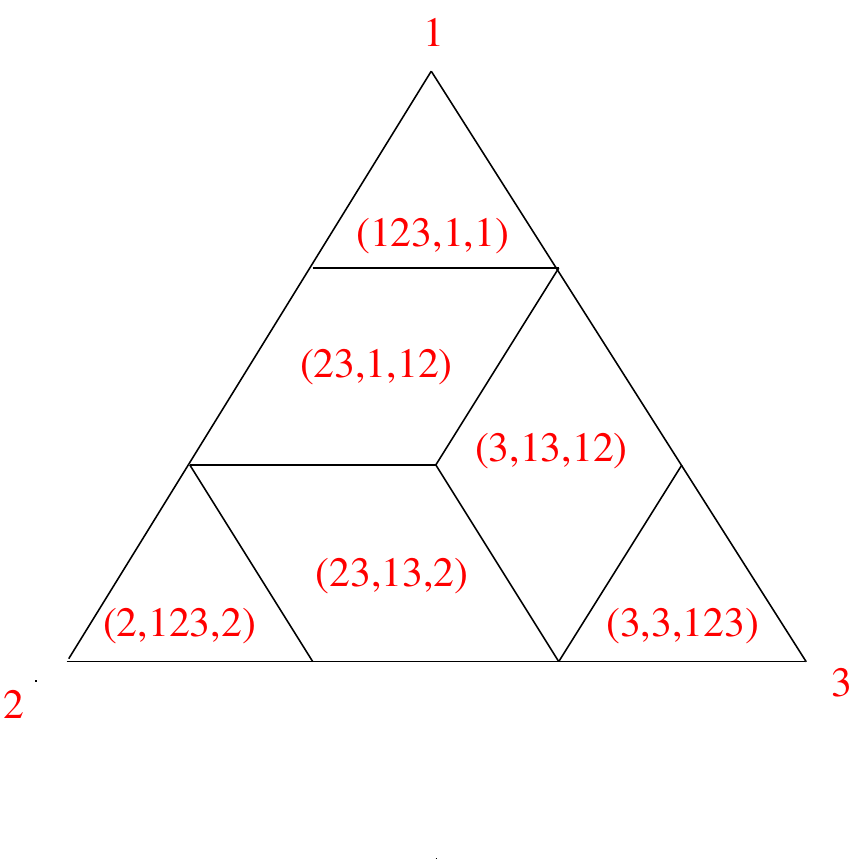}
	\caption{The set of topes of a tropical oriented matroid, and the set of tree-types of the same tropical oriented matroid.}
	\label{fig:msub-trees}
\end{figure}



Let $S$ be a collection of $(n,d)$-topes. We say that $S$ satisfies the \newword{tope-linkage-property}, if for any tope $P$ in $S$ and any $q \in [d]$ such that $P \not = (q,\ldots,q)$, there is a tope $P'$ in $S$ that is obtained from $P$ by replacing some element $t \not = q$ with $q$ in some coordinate of $P$. In such case, we express this as $P \stackrel{-t+q}\rightarrow P'$.


\begin{lemma}
\label{lem:topelink}
Let $\OO$ be a tropical oriented matroid (with parameters $(n,d)$) and let $\hat{\OO}$ denote the set of topes of $\OO$. Then $\hat{\OO}$ satisfies the tope-linkage-property.
\end{lemma}
\begin{proof}
Let $P$ be a tope such that $P \not = Q := (q,\ldots,q)$. Define $dist(P,Q)$, the distance between two topes $P,Q$ to be the number of coordinates such that $P_i \not = Q_i$. If $dist(P,Q) = 1$, there is nothing left to show. Hence assume that $dist(P,Q) > 1$. 

We use the exchange axiom between $P$ and $Q^0 :=Q$ and then use refinement to find a tope $Q^1 \not = P,Q^0$ such that $Q^1_i$ is either equal to $P_i$ or $\{q\}$. Then $dist(P,Q^1) < dist(P,Q^0)$. If $dist(P,Q^1) = 1$, there is nothing left to show. In the other case, repeat the same procedure for $P$ and $Q^1$ to get a new tope $Q^2$. Repeat this until we get some $Q^k$ with $dist(P,Q^k)=1$. This process has to end in finite steps, since the distance goes down by at least $1$ every time.
\end{proof}

Take a look at Figure~\ref{fig:msub-topes-v01}. Let $P$ be the tope $(3,1,2)$ and choose $q$ to be $2$, since $(3,1,2) \not = (2,2,2)$. We can find a tope $(2,1,2)$, to check that the linkage-property holds. We can express this as $(3,1,2) \stackrel{-3+2}\rightarrow (2,1,2)$.



Just like matroids, we can define restriction, contraction and dual operations on tropical oriented matroids, to get another tropical oriented matroid.

\begin{proposition}[\cite{2007arXiv0706.2920A}]
Let $\OO$ be a tropical oriented matroid with parameters $(n,d)$. Pick any set of coordinates for $I \subseteq [n]$. Then the \newword{restrction} $\OO|_{I,[d]}$, which consists of all types of $\OO$ by deleting coordinates $i\not \in I$, is also a tropical oriented matroid. Pick any set of directions $J \subseteq [d]$. Then the \newword{contraction} $\OO|_{[n],J}$, which consists of all types of $\OO$ which do not contain elements outside of $J$ in any coordinate, is also a tropical oriented matroid. The \newword{minor} $\OO_{I,J}$, which consists of all types of $\OO$ by deleting coordinates $i \not \in I$ and choosing the types that only use elements of $J$, is also a tropical oriented matroid.
\end{proposition}

Given a type $A$, we denote the type obtained from $A$ by deleting coordinates $i \not \in I$, and deleting the elements not contained in $J$, as $A|_{I,J}$.

Let $\OO$ be a tropical oriented matroid encoding a triangulation of $\Delta_{n-1} \times \Delta_{d-1}$. This triangulation induces a triangulation of $\Delta_I \times \Delta_J$. The tropical oriented matroid encoding this induced triangulation is the minor $\OO|_{I,J}$ of $\OO$.

Let us look at the tropical oriented matroid $\OO$ represented in Figure~\ref{fig:msub-trees}. Given a type $(3,1,2)$, its restriction to $\{1,2\}$ is $(3,1,2)|_{\{1,2\},[3]} = (3,1)$. The minor $\OO_{\{1,2\},\{1,3\}}$ consists of types $\{(1,1),(3,1),(3,3),(3,13),(13,1)\}$.


\begin{definition}
A \newword{semitype} (with parameters $(n,d)$) is given by an n-tuple of subsets of $[d]$, not
necessarily nonempty. Given a tropical oriented matroid $\OO$, its completion $\tilde{\OO}$ consists of all semitypes which result from types of $\OO$ by changing some subset of the coordinates to the empty
set. Given a collection of semitypes, its \newword{reduction} consists of all honest types contained in the collection.
\end{definition}
\begin{definition}
Let $A$ be a semitype (with parameters $(n,d)$). Then the transpose $A^T$ of $A$, a semitype with parameters $(d,n)$ ( i.e. a $d$-tuple of subsets of $[n]$), has $i \in A^T_j$ whenever $j \in A_i$.
\end{definition}

The transpose of the type $(2,1,1)$ is the semitype given by $(23,1,\emptyset)$.

\begin{theorem}[\cite{2007arXiv0706.2920A},\cite{MR2957992},\cite{MR2820754}]
Let $\OO$ be a tropical oriented matroid. Then the dual of $\OO$, which is denoted as $\OO^T$, is the reduction of the collection of semitypes given by transposes of semitypes in $\tilde{\OO}$, which is also a tropical oriented matroid.
\end{theorem}

Let $\OO$ be a tropical oriented matroid encoding a triangulation of $\Delta_{n-1} \times \Delta_{d-1}$. This triangulation can also be thought of as a triangulation of $\Delta_{d-1} \times \Delta_{n-1}$. The tropical oriented matroid encoding that triangulation, is the dual of $\OO$.

We end the section with 2 lemmas that will be needed for the main proof. Given two $(n,d)$-types $A$ and $B$, we define their \newword{union} to be a type $C$ where $C_i = A_i \cup B_i$ for all $i \in [n]$.

\begin{lemma}[\cite{MR2957992}]
\label{lem:samecoordunion}
Let $\OO$ be a tropical oriented matroid and $j$ be some element of $[d]$. Let $A$ and $B$ be types of $\OO$ such that $A_i = B_i$ for all $i \not = j$. Then the union of $A$ and $B$ is also a type of $\OO$.
\end{lemma}

\begin{lemma}
\label{lem:diffcoordunion}
Let $\OO$ be a tropical oriented matroid and $q,t$ be some elements of $[d]$. Let $B$ and $C$ be types of $\OO$, such that $B$ can be obtained from $C$ by deleting all occurances of $q$ in $C$, and then by adding $t$'s to some coordinates. Then the union of $B$ and $C$ is also a type in $\OO$.

\end{lemma}

\begin{proof}
We will show that there is a cell which contains both $B$ and $C$. Assume for the sake of contradiction that there is no such cell. This means that there is some hyperplane $\sum_{i \in I} x_i = c$ which contains the common refinement of $B$ and $C$, but separates $B$ and $C$. This is impossible, since $q$ only appears in $C$ and $t$ only appears in $B$.

\end{proof}

\section{Matching Ensemble}

In this section, we will define matching ensembles. To do so, we will borrow the notion of \newword{matching fields} and \newword{linkage axiom} which was used in \cite{MF}. Given two sets $A$ and $B$ of equal cardinality, the \newword{matching} is a bijection between $A$ and $B$. We think of this as a bipartite graph, with left vertex set $A$ and right vertex set $B$, and edge set given by the set of edges $(a,\pi(a))$ where $\pi$ is the bijection between $A$ and $B$. Although the matching fields used in \cite{MF} only concerns $d$-by-$d$ matchings, we extend the definition and look at matchings of all sizes.


\begin{definition}
We say that a collection $\M$ of matchings between subsets of $[n]$ and subsets of $[d]$ forms a \newword{matching field} (with parameters $(n,d)$) if it satisfies the following two axioms:
\begin{itemize}
\item There is exactly one matching for any pair $I \subseteq [n], J \subseteq [d]$ such that $|I|=|J|$.
\item Let $M$ be a matching between $I$ and $J$. Let $M'$ be a matching obtained by taking a subgraph of $M$. Then $M'$ is also in $\M$.
\end{itemize}

If a matching field $\M$ also satisfies the following axioms, we call it a \newword{matching ensemble}:

\item(left linkage) Let $M$ be a matching between $I$ and $J$ in $\M$. Pick any $v \in L \setminus [n]$. Then there is an edge $(i,j) \in M$ that we can replace with $(v,j)$ to get another matching $M'$ in $\M$. 

\item(right linkage) Let $M$ be a matching between $I$ and $J$ in $\M$. Pick any $v \in R \setminus [d]$. Then there is an edge $(i,j) \in \M$ that we can replace it with $(i,v)$ to get another matching $M'$ in $\M$. 

\end{definition}

We will note here that in \cite{MF}, the term \newword{linkage axiom}, a combination of the left linkage and right linkage axioms, is used instead. The reason we split the linkage axiom in this paper, is to use duality in the proofs. An example of a matching ensemble is given in Figure~\ref{fig:ME-ex}.

\begin{figure}[htbp]
	\centering
		\includegraphics[width=0.65\textwidth]{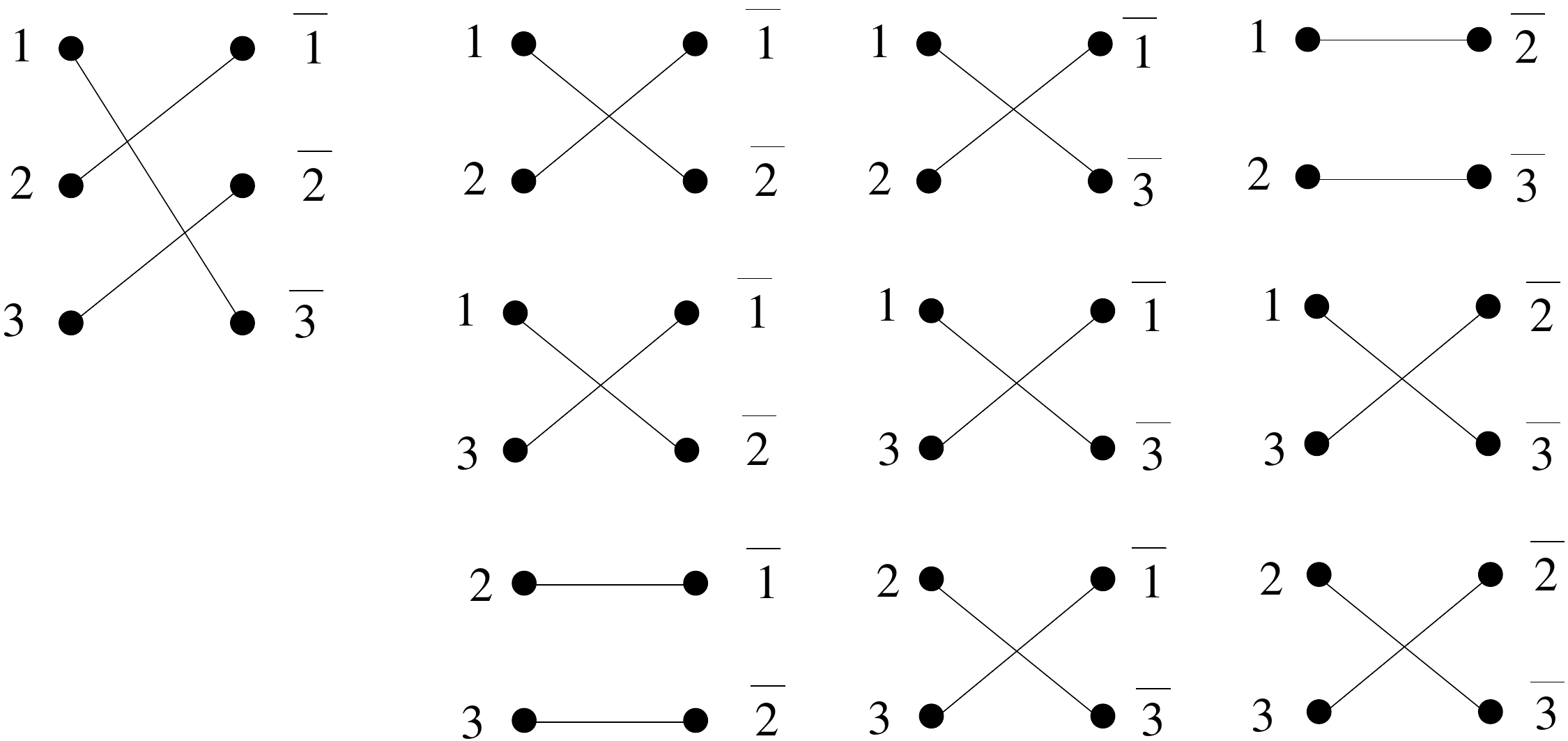}
	\caption{A matching ensemble.}
	\label{fig:ME-ex}
\end{figure}

Let $\M$ be a matching with parameters $(n,d)$. By swapping $[n]$ and $[d]$, we get a matching field $\M'$ with parameters $(d,n)$, which we call the \newword{dual} of $\M$. It is easy to see that $\M$ is a matching ensemble if and only if $\M'$ is a matching ensemble.

We introduce an \newword{extraction method} that extracts a collection of matchings from an $(n,d)$-tropical oriented matroid $\OO$. Let $I$ be any subset of $[n]$ and $J$ be any subset of $[d]$ such that $|I| = |J|$. Take the minor $\OO|_{I,J}$. Lemma~\ref{lem:topesoftom} tells us that there is a unique tope $A$ in $\OO$ such that $pos(A) = (1,\ldots,1)$. We call such tope a \newword{central tope} of the tropical oriented matroid $\OO|_{I,J}$. This tope gives a matching between $I$ and $J$ via $G_A$. If we extract a matching for all pairs $(I \subseteq [n],J \subseteq [d])$ such that $|I|=|J|$, we get a matching field (with parameters $(n,d)$). We denote this matching field coming from the tropical oriented matroid $\OO$ as $\M_{\OO}$.

Since $(\OO|_{I,J})^T = \OO^T|_{J,I}$ for any tropical oriented matroid $\OO$:

\begin{lemma}
\label{lem:MEdual}
Let $\OO$ be a tropical oriented matroid. Then $\M_{\OO^T}$ is the dual matching field of $\M_{\OO}$.
\end{lemma}

Using this lemma, we will show that $\M_{\OO}$ is a matching ensemble.

\begin{proposition}
\label{prop:tomtome}
If we use the extraction method on an $(n,d)$-tropical oriented matroid, we get a matching ensemble.
\end{proposition}

\begin{proof}
It is enough to show that the matching field we get from an arbitrary tropical oriented matroid satisfies the right-linkage axiom. This is due to the following reasoning : Let $\OO$ be a tropical oriented matroid. If we can show that the right-linkage axiom holds for $\M_{\OO^T}$, it implies that the left-linkage axiom holds for $\M_{\OO}$.


Let us look at a matching between $I$ and $J$ where $I \subseteq [n],J \subseteq [d]$ and $|I| = |J|$. This matching corresponds to the central tope $T$ of the tropical oriented matroid $\OO_{I,J}$. Now consider the minor $\OO_{I,[d]}$. $T$ has to be a tope inside this tropical oriented matroid. Let $v$ be an element of $J^c$. Via the tope-linkage-property, there exists some tope $T'$ and $t \in [d]$ such that $T \stackrel{-t+v}\rightarrow T'$. This is the unique central tope of the minor $\OO_{I,J \setminus \{t\} \cup \{v\}}$. Hence the right-linkage axiom holds for a matching field coming from an arbitrary tropical oriented matroid.
\end{proof}

We define the $(I,J)$-\newword{minor}, where $I \subseteq [n]$ and $J \subseteq [d]$, of a matching ensemble $\M$ (with parameters $(n,d)$) to be the collection of all matchings between subsets of $I$ and subsets of $J$ in $\M$. We denote this minor as $\M_{I,J}$, and this collection is also indeed a matching ensemble. Since each matching inside $\M$ is obtained by studying a minor of $\OO$, we get the following result: 

\begin{lemma}
\label{lem:MEminor}
Let $\OO$ be a tropical oriented matroid and $\M$ be a matching ensemble such that $ext(\OO) = \M$. Then $ext(\OO|_{I,J}) = \M|_{I,J}$.
\end{lemma}

Now we show that we get different matching ensembles from different tropical oriented matroids.

\begin{lemma}
\label{lem:extuniq}
Let $\OO$ be a tropical oriented matroid and let $\M$ be a matching ensemble such that $ext(\OO) = \M$. If $\OO'$ is a tropical oriented matroid such that $ext(\OO') = \M$, we have $\OO = \OO'$.
\end{lemma}
\begin{proof}
Assume for the sake of contradiction that $\OO \not = \OO'$. Due to Theorem~\ref{thm:topestom} and Lemma~\ref{lem:topesoftom}, there are topes $A \in \OO$ and $B \in \OO'$ such that $pos(A) = pos(B)$ and $A \not = B$. Since $pos(A) = pos(B)$, $B$ can be obtained from $A$ by permuting the coordinates. After crossing out all coordinates such that $A_i = B_i$, we can find a cycle between $A$ and $B$. Let $I$ be the set of coordinates that is involved in this cycle, and let $J$ be the set of elements of $[d]$ involved in this cycle. Then $A|_{I,J}$ and $B|_{I,J}$ each describe a different matching between $I$ and $J$, where both of them have to be in $\M$. Hence we get a contradiction.
\end{proof}

\section{Matching Ensembles and triangulation of $\Delta_{n-1} \times \Delta_{d-1}$}

In the previous section, we have shown that we can get a matching ensemble from a triangulation of $\Delta_{n-1} \times \Delta_{d-1}$. In this section, we show the other direction, that one can construct the triangulation back from the given matching ensemble.

We first start out with several tools.

Let $P$ be an $(n,d)$-tope and $T$ be a tree-$(n,d)$-type where $n \leq d$. We use $supp(P)$ to denote the set of elements of $[d]$ that occur at least once in $P$.

\begin{lemma}
\label{lem:t1}
If there is a length $n$ cycle between $T$ and $P$ that is of minimal length, then there can't be an element $i$ of  $supp(P)$ that occurs more than $2$ times in $T$. And not all elements of $supp(P)$ can occur twice in $T$.
\end{lemma}
\begin{proof}
Without loss of generality, we may assume that $p_2 \in T_1,\ldots,p_n \in T_{n-1},p_1 \in T_1$ where we use $p_i$ to denote the lone element of $P_i$. For sake of contradiction, assume $p_1$ occurs at least $3$ times in $T$. Which means there is some $j \not =1,n$ such that $p_1 \in T_j$. This implies that there is a length $j < n$ cycle between $T$ and $P$.

For the second claim, assume for the sake of contradiction that all elements of $supp(P)$ occur twice in $T$. This implies $|T_1| + \cdots + |T_n|$ is at least $2n + (d-n) = n + d$, from which we get a contradiction since $G_T$ has to be a tree of $K_{n,d}$, and has exactly $n+d-1$ edges.

\end{proof}


\begin{lemma}
\label{lem:t2}
Assume there is a length $n$ cycle between $T$ and $P$ that is of minimal length. Let $q$ be an element of $[d]$ that occurs at least twice in $T$. Then $P'$ and $T$ are also incompatible, where $P'$ is obtained from $P$ via the linkage, in a way that $P \stackrel{-t+q}\rightarrow P'$. Moreover, if the minimal length of cycle between $T$ and $P'$ is also $n$, then $t$ can occur at most once in $T$.
\end{lemma}
\begin{proof}
Without loss of generality, we may assume that $p_2 \in T_1,\ldots,p_n \in T_{n-1},p_1 \in T_1$ where we use $p_i$ to denote the lone element of $P_i$. We may also assume that $P'$ is obtained from $P$ by switching $p_1$ to $q$. Now since $q$ occurs twice in $T$, there exists $j \not = 1$ such that $q \in T_j$. This implies that we get a cycle of length $j$ between $T$ and $P$.

Now for the second claim, we must have $j=n$ in the previous cycle. Assume for the sake of contradiction that $t$ occurs at least twice. For the minimal length cycle between $T$ and $P'$ to have length $n$, we must have $t=p_1 \not \in T_2,\ldots,T_{n-1}$. This implies that $t=p_1 \in T_1,T_n$. But we also have $q \in T_1,T_n$. Since $p$ and $q$ are both contained in $T_1$ and $T_n$, this contradicts the fact that $G_T$ is a spanning tree of $K_{n,d}$.
\end{proof}

Using the above lemmas, we will show that a matching ensemble describes a triangulation of $\Delta_{n-1} \times \Delta_{d-1}$.

\begin{proposition}
\label{prop:metotom}
Let $\M$ be an $(n,d)$-matching ensemble. Then there is a triangulation $\T$ of $\Delta_{n-1} \times \Delta_{d-1}$ such that $ext(\T) = \M$.
\end{proposition}

\begin{proof}
We use induction on $n$. If $n=2$, we can think of each $2$-by-$2$ matchings in a matching ensemble $\M$ as putting an ordering on $[d]$. That is, if we have $(1,\bar{i})$ and $(2,\bar{j})$ in a matching between $[2]$ and $\{\bar{i},\bar{j}\}$, we say that $\bar{i} < \bar{j}$. And this gives a total ordering on $[d]$. For the sake of contradiction, assume we get $\bar{i_1} < \cdots < \bar{i_k}$ and $\bar{i_k} < \bar{i_1}$, and let $k$ be minimal among such cyclic relationships. Using the right linkage axiom, from the matching consisting of edges $(1,\bar{i_k})$ and $(2,\bar{i_1})$, we must get a matching that either
\begin{itemize}
\item consists of edges $(1,\bar{i_2})$ and $(2,\bar{i_1})$, or
\item consists of edges $(1,\bar{i_k})$ and $(2,\bar{i_2})$.
\end{itemize}

In the first case we get $\bar{i_2} < \bar{i_1}$, and in the second case we get $\bar{i_k} < \bar{i_2}$, where in both of the cases we get a contradiction. So we get a total ordering, and reorder the elements of $[d]$ such that $\bar{1} < \cdots < \bar{d}$. Now we construct tree-$(2,d)$-types $T^i := (\{1,\ldots,i\},\{i,\ldots,d\})$. The types are pairwise compatible, and by Proposition~\ref{prop:triangfill}, we get a triangulation $\T$ of $\Delta_{n-1} \times \Delta_{d-1}$. Since all the $2$-by-$2$ matchings are subgraphs of $T^i$'s, we have $ext(\T) = \M$. Moreover, any tree-$(2,d)$-type that is not one of the $T^i$'s is not compatible with the $(2,d)$-topes coming from $2$-by-$2$ matchings.

Hence we have proven the claim for $n=2$, the base case. Assume for the sake of induction that it is true for smaller values of $n$.

Now we use induction on $d$. If $d=2$, the result follows from the exactly same reasoning as $n=2$ case. When we have $d < n$, the result follows from induction hypothesis by taking the dual and using Lemma~\ref{lem:MEdual}. Hence we only need to consider the case when $d \geq n$. We assume for the sake of induction that the claim is true for smaller values of $d$.

Induction hypothesis and Lemma~\ref{lem:extuniq} tells us that there is a unique tropical oriented matroid $\OO_{I,J}$ for $\M|_{I,J}$, as long as $(I,J) \not = ([n],[d])$. Collect all topes of $\OO_{[n],[d] \setminus \{i\}}$ for each $i \in [d]$, to form a collection $\PP$. If $n=d$, also add the tope corresponding to the matching between $[n]$ and $[d]$.

If some tope $A \in {\OO_{[n],[d] \setminus \{i\}}}$ and $B \in {\OO_{[n],[d] \setminus \{j\}}}$ have the same position in $n\Delta_{[d]}$, it implies that $supp(A)=supp(B) \subseteq [d] \setminus \{i,j\}$, and hence $A$ and $B$ are topes of $\OO_{[n],[d] \setminus \{i,j\}}$, due to Lemma~\ref{lem:MEminor}. Then, Lemma~\ref{lem:topesoftom} tells us that $A$ and $B$ should be the same. Therefore, there is a bijection between $\PP$ and $H_{n,d}$.

Now let us see that the topes of $\PP$ are pairwise compatible. Assume for the sake of contradiction there are some topes $P$ and $P'$ that are not compatible. If $supp(P)=supp(P')$, then they are both topes of $\OO|_{[n],S}$ for some proper subset $S$ of $[d]$ and must be compatible. If not, then the length of the minimal cycle between $P$ and $P'$ has to be smaller than $n$. Let $I$ be the subset of $[n]$ which contains the coordinates involved in that cycle. By restricting $[n]$ to $I$, we get two incompatible topes in the tropical oriented matroid coming from the matching field $\M|_{I,[d]}$, and we get a contradiction.

Now for each unit lattice simplex, take the union of the topes corresponding to the base points of the simplex. We will show that this gives a spanning tree. Pick a unit simplex, which has right degree vector $a$. Denote the tope having position vector $a + e_i$ as $P^i$, define $\SSSS$ to be the collection of such topes. We denote $\SSSS_I$ to denote the set consisting of $P^i$'s for $i \in I$.

Assume there exists $i \not = j$ such that $a_i = a_j = 0$. By induction hypothesis, the set of topes $\SSSS_{[d] \setminus \{i\}},\SSSS_{[d] \setminus \{j\}},\SSSS_{[d] \setminus \{i,j\}}$ are topes of a tropical oriented matroid corresponding to $\M_{[n],[d] \setminus \{i\}},\M_{[n],[d] \setminus \{j\}},\M_{[n],[d] \setminus \{i,j\}}$ respectively. Each of their union forms spanning trees $T^i,T^j,T^{i,j}$ of $K_{[n],[d] \setminus \{i\}},K_{[n],[d] \setminus \{j\}},K_{[n],[d] \setminus \{i,j\}}$ respectively. Now since $T^{i,j}$ is a subgraph of both $T^i$ and $T^j$, the union of $T^i$ and $T^j$, which is the union of the topes of $\SSSS$, has to be a spanning tree of $K_{[n],[d]}$.

What remains is the case when there is only one $i$ such that $a_i=0$. Since $n \leq d$, this means that $n=d$ and $pos(P^i) = (1,\ldots,1)$. Let $j$ be the coordinate such that $P_j^i = \{i\}$. The tope $A$ we get from $P^i$ be removing the $j$-th coordinate, corresponds to a matching which is already in $\M$, and hence there is a tope $P'$ of $\OO_{[n],[d] \setminus \{i\}}$ which restricts to $A$. Since $pos(P') = a + e_t$ for some $t$, we have $P' = P^t$. Corollary~\ref{cor:tombaseunion} tells us that the union of topes in $\SSSS_{[d] \setminus \{i\}}$ forms a spanning tree of $K_{[n], [d] \setminus \{i\}}$. Since this tree contains $G_{P'}$, the union of topes in $\SSSS$, which is same as the union of this tree and $G_{P^i}$, forms a spanning tree of $K_{[n],[d]}$.

Therefore, we get a spanning tree for each unit lattice simplex of $n\Delta_{d-1}$, and we will denote this set as $\SSS$. We will now show that the trees of $\SSS$ are compatible with the topes of $\PP$.

Our first step is to show that if $T$ and $P$ are incompatible, then the length of the cycle between $T$ and $P$ has to be $n$. To show this, we will prove that $P|_{I,[d]}$ and $T|_{I,[d]}$ are types in $\OO_{I,[d]}$ for an arbitrary subset $I$ of $[n]$. First for $P$, if $P \in \OO_{[n],[d] \setminus \{i\}}$ for some $i \in [d]$, we know that $P|_{I,[d]}$ is a tope of $\OO_{I,[d] \setminus \{i\}}$, which is a subset of $\OO_{I,[d]}$. If not, then we have $n=d$ and $P$ is the tope corresponding to the matching between $[n]$ and $[d]$ in $\M$. Since a restriction of a matching is also contained in the ensemble, it follows that $P|_{I,[d]}$ is a tope of $\OO_{I,[d]}$.

Next, we look at $T$. Let $a$ be the right degree vector of $T$. First consider the case when there is some $i$ and $j$ such that $a_i = a_j = 0$. Let $A$ be a type of $\OO_{[n],[d] \setminus \{i\}}$ obtained by deleting $i$ from $T$. Similarly, let $B$ be a type of $\OO_{[n],[d] \setminus \{j\}}$ obtained by deleting $j$ from $T$. We know that $A|_{I,[d]}$ and $B|_{I,[d]}$ are types in $\OO_{I,[d]}$. Hence by Lemma~\ref{lem:diffcoordunion}, we know that the union of $A|_{I,[d]}$ and $B|_{I,[d]}$, which is exactly $T|_{I,[d]}$, is a type in $\OO_{I,[d]}$. For the other case, when there is only one $i$ such that $a_i=0$, there has to be some $j$ such that 
\begin{itemize}
\item $j$ appears twice in $T$ and,
\item there is a $k \in [n]$ such that $T_k = \{j\}$.
\end{itemize}

Without loss of generality, let us assume that $k=1$. Let $A$ be a type of $\OO_{[n],[d] \setminus \{i\}}$ obtained by deleting $i$ from $T$. We know that there is a type $B$ in $\OO_{[n] \setminus \{1\},[d]}$ that restricts to $T|_{I \setminus \{1\},[d]}$. Using exchange in $\OO_{I,[d]}$, we know that a type $C$, which is given by $C_1 = A_1 \cup B_1$ and $C_i = A_i$ for all $i \not =1$ is a type in $\OO_{I,[d]}$. Now using Lemma~\ref{lem:diffcoordunion} between $B$ and $C$, we can see that $T|_{I,[d]}$ is a type in $\OO_{I,[d]}$.





Therefore, if there is a length $k<n$ cycle between $T$ and $P$, let $i$ be the coordinate that is not involved in the cycle. We have shown in the previous paragraphs that $T|_{[n] \setminus \{i\},[d]}$ and $P|_{[n] \setminus \{i\},[d]}$ are types of $\OO_{[n] \setminus \{i\},[d]}$, and they must be incompatible, hence we get a contradiction. Therefore, any cycle between $T \in \SSS$ and $P \in \PP$ has to have length equal to $n$. This also implies that $P$ has to be a tope corresponding to a matching in $\M$.

If there is some $q \in [d] \setminus supp(P)$ that appears at least twice in $T$, the right linkage property of $\M$ tells us that there is some tope $P' \in \PP$ such that $P \stackrel{-t+q}\rightarrow P'$, where $t$ is some element of $[d]$. By Lemma~\ref{lem:t2}, $P'$ and $T$ are also incompatible. Cycle between $P'$ and $T$ also has to have length $n$, and again by using Lemma~\ref{lem:t1} and Lemma~\ref{lem:t2}, $t$ has to occur exactly once in $T$. Therefore, we can repeat this process till we get a tope $P''$, which is incomparable with $T$, and every element of $[d] \setminus supp(P'')$ appears exactly once in $T$. Set this $P''$ as our new $P$. In such case, Lemma~\ref{lem:t1} implies that $pos(P)$ is a base point of the unit simplex located at $RD(T)$. Then $P$ is the tope used in the union to construct $T$. This contradicts the fact that $P$ and $T$ are incompatible.

Therefore, we have shown that any $T \in \SSS$ and any $P \in \PP$ are compatible. This implies that any tope we can get by refining some $T \in \SSS$ must be in $\PP$. If $T,T' \in \SSS$ are incompatible, we can find a tope $P$ by refining $T'$ such that $T$ and $P$ are incompatible. Therefore, any pair of spanning trees in $\SSS$ are pairwise compatible. Since we have a spanning tree for each unit lattice simplex of $n\Delta_{d-1}$, Lemma~\ref{prop:triangfill} tells us that we get a triangulation, and hence a tropical oriented matroid $\OO$ such that $ext(\OO) = \M$.

\end{proof}

The above proposition, along with Proposition~\ref{prop:tomtome} proves the main result of the paper:

\begin{theorem}
\label{thm:main}
There is a bijection between matching ensembles (with parameters $(n,d)$) and triangulations of $\Delta_{n-1} \times \Delta_{d-1}$.
\end{theorem}

\begin{question}
What would be the matching ensemble analogue for subdivisions of $\Delta_{n-1} \times \Delta_{d-1}$?
\end{question}


\bibliographystyle{plain}    
\bibliography{rs}        
 
\end{document}